\documentclass[11pt]{article}

\usepackage{amssymb, amsmath, fullpage, amsthm, pictexwd}
\newtheorem{theorem}{Theorem}[section]
\newtheorem{lemma}[theorem]{Lemma}
\newtheorem{proposition}[theorem]{Proposition}
\newtheorem{corollary}[theorem]{Corollary}

\newtheorem{example}[theorem]{Example}

\makeatletter
\@addtoreset{equation}{section}
\makeatother

\newcommand{\SSS}{{\mathfrak S}}
\newcommand{\wt}{\operatorname{wt}}
\newcommand{\Wt}{\operatorname{Wt}}
\newcommand{\pair}[2]{\left(#1,#2\right)}
\newcommand{\av}{{\bf a}}
\newcommand{\bv}{{\bf b}}
\newcommand{\ab}{\av\bv}

\newcommand{\nint}[1]{\left\lfloor #1 \right\rceil}

\renewcommand{\qedsymbol}{
{\nobreak \ifvmode \relax \else
 \ifdim\lastskip<1.5em \hskip-\lastskip
 \hskip1.5em plus0em minus0.5em \fi \nobreak
 \vrule height0.75em width0.5em depth0.25em\fi}
}
\begin{document}

\author{Richard Ehrenborg and JiYoon Jung}
\title{Descent pattern avoidance}
\date{}

\maketitle

\begin{abstract}
We extend the notion of consecutive pattern
avoidance to considering
sums over all permutations where each term
is a product of weights depending on each consecutive pattern
of a fixed length.
We study the problem of finding the asymptotics of these sums.
Our technique is to extend the spectral method
of Ehrenborg, Kitaev and Perry.
When the weight depends on the descent pattern
we show how to find the equation
determining the spectrum.
We give two length $4$ applications.
First, we find the asymptotics of the number of permutations
with no triple ascents and no triple descents.
Second, we give the asymptotics of the number of permutations
with no isolated ascents or descents.
Our next result is a weighted pattern of length $3$
where the associated operator only has one non-zero
eigenvalue. Using generating functions we show
that the error term in the asymptotic expression
is the smallest possible.
\end{abstract}

\section{Introduction}

Ehrenborg, Kitaev and Perry~\cite{Ehrenborg_Kitaev_Perry}
used the spectrum of linear operators on the space $L^{2}([0,1]^{m})$
to study the asymptotics of consecutive pattern avoidance.
We extend their techniques to study
asymptotics of sums over all permutations
where each term is a product of weights which depend on
the consecutive patterns of a fixed length $m+1$.
When the weights are all zero or one, this reduces to
studying consecutive pattern avoidance.
Furthermore, when the weights depend on
the descent pattern, we show how
to obtain the equation whose roots are the spectrum
of the associated linear operator.
In general this is a
transcendental equation.

We give two length $4$ examples. First we study
the number of permutations with no triple ascents and
no triple descents. This is equivalent to
$\{1234,4321\}$-avoiding permutations.
We determine the transcendental eigenvalue equation
and a numerical approximation to the largest root, 
which gives the asymptotics of the number such permutations.

The second example is
permutations that avoid the ten alternating patterns
$1324$, $1423$, $2314$, $2413$, $3412$ and
$2143$, $3142$, $3241$, $4132$, $4231$.
This is the class of permutations with no isolated ascents or descents.
Yet again, we obtain
the transcendental eigenvalue equation satisfied by
the spectrum and give numerical approximation to its largest root.

We next turn to a weighted length $3$ example.
We are interested in the sum over all $123$-avoiding permutations
where the term is $2$ to the power of the number
of double descents.
Here we also consider the extra conditions if the permutation
begins/ends with an ascent or a descent.
The associated operator only has one non-zero eigenvalue,
namely $1$. Hence the asymptotics is a constant $c$ times $n$ factorial
and the error term is bounded by $n! \cdot r^{n}$
where~$r$ is an arbitrary small positive number.

It remains to understand the error term. We are able to find
the associated generating functions. Furthermore, we show that
the error term is the smallest
possible! The asymptotics is $c \cdot n!$
(where the constant $c$ is irrational, in fact, transcendental)
and the explicit expression is the nearest integer to $c \cdot n!$
for large enough $n$.
This behavior also occurs with the derangement numbers.
This classical
sequence makes its appearance as one of the sequences that we study.

We end the paper with concluding remarks and open problems.

\section{Weighted consecutive pattern avoidance}

For $x_{1}, x_{2}, \ldots, x_{k}$
distinct real values, define
$\Pi(x_{1}, x_{2}, \ldots, x_{k})$ to be the unique
permutation $\sigma$ in the symmetric group $\SSS_{k}$ such that
$x_{i} < x_{j}$ if and only if $\sigma_{i} < \sigma_{j}$
for all indices $1 \leq i < j \leq k$.
We say that a permutation $\pi$ in $\SSS_{n}$
consecutively avoid
a permutation $\sigma$ in $\SSS_{m}$ if there
is no index $i$ such that
$\Pi(\pi_{i}, \pi_{i+1}, \ldots, \pi_{i+m-1}) = \sigma$.

Let $\wt$ be
a real-valued weight function on the symmetric group
$\SSS_{m+1}$.
Similarly, let $\wt_{1}, \wt_{2}$ be
two real-valued weight functions on the symmetric group
$\SSS_{m}$.
We call $\wt_{1}$ and $\wt_{2}$ the initial,
respectively, the final weight function.
We extend these three weight functions to
the symmetric group $\SSS_{n}$ for $n \geq m$ by
defining
\begin{eqnarray*}
     \Wt(\pi)
  & = &
       \wt_{1}(\Pi(\pi_{1}, \pi_{2}, \ldots, \pi_{m})) \\
  &   &
             \cdot
     \prod_{i=1}^{n-m}
       \wt(\Pi(\pi_{i}, \pi_{i+1}, \ldots, \pi_{i+m})) \\
  &   &
             \cdot
       \wt_{2}(\Pi(\pi_{n-m+1}, \pi_{n-m+2}, \ldots, \pi_{n})) .
\end{eqnarray*}
In other words, the weight of a permutation $\pi$
in $\SSS_{n}$
is the product of the initial weight function $\wt_{1}$
applied to the $m$ first entries of $\pi$ with
the product of the weight function $\wt$ applied to
every segment of $\pi$ of length $m+1$
with the final weight function $\wt_{2}$
applied to the $m$ last entries of $\pi$.
The question is what can one say about the quantity
$$ \alpha_{n}
       =
   \sum_{\pi \in \SSS_{n}} \Wt(\pi)  . $$

Consecutive pattern avoidance can be studied this way
by using the weight functions
$\wt_{1}(\sigma) = \wt_{2}(\sigma) = 1$ 
for all $\sigma$ in $\SSS_{m}$ and
$\wt(\sigma) = 1$ if $\sigma \not\in S$ and
$\wt(\sigma) = 0$ otherwise,
where $S \subseteq \SSS_{m+1}$
is the set of forbidden patterns.
Observe then that a permutation $\pi \in \SSS_{n}$ avoids the 
patterns in~$S$ if and only if $\Wt(\pi) = 1$.
Note that by letting
the initial weight function $\wt_{1}$ and
the final weight function $\wt_{2}$ be $0,1$-functions,
we are studying
consecutive pattern avoidance with forbidden initial and
final configurations.

The methods of Ehrenborg, Kitaev and Perry~\cite{Ehrenborg_Kitaev_Perry}
to study the asymptotics of
consecutive pattern avoidance by considering the spectrum
of operators on $L^{2}([0,1]^{m})$ naturally extend to
this more general setting of weights on permutations.

Define the function $\chi$ on
the $(m+1)$-dimensional unit cube $[0,1]^{m+1}$ by
$\chi(x)  =  \wt(\Pi(x))$.
Note that $\chi$ is
undefined on a point with two equal coordinates.
However, this situation occurs on a set
of measure zero and hence can be ignored.
Next define the operator $T$ on 
the space $L^{2}([0,1]^{m})$ by
\begin{equation}
    T(f(x_{1},\ldots,x_{m}))
  =
    \int_{0}^{1}
        \chi(t,x_{1},\ldots,x_{m})
      \cdot
        f(t,x_{1},\ldots,x_{m-1}) \: dt  .
\label{equation_T}
\end{equation}
Note that $L^{2}([0,1]^{m})$ is a Hilbert space
with the inner product defined by
$$    \pair{f}{g}
   =
      \int_{[0,1]^{m}} 
          f(x_{1}, \ldots, x_{m})
        \cdot
          \overline{g(x_{1}, \ldots, x_{m})}
             \:
        dx_{1} \cdots dx_{m}  .  $$
The adjoint operator $T^{\ast}$ is defined by the relation
$\pair{f}{T^{\ast}(g)} = \pair{T(f)}{g}$. For the operator
$T$ defined in equation~(\ref{equation_T}) we have that
$$
T^{\ast}(f(x_{1},\ldots,x_{m}))
    =
\int_{0}^{1} \chi(x_{1},\ldots,x_{m},u) \cdot f(x_{2},\ldots,x_{m},u) \: du
$$
Finally, the spectrum of an operator $T$
is all the values $\lambda$ such that
$T - \lambda \cdot I$ is not an invertible operator.

Similarly to the function $\chi$,
define
the two functions $\kappa$ and $\mu$ on
the $m$-dimensional unit cube $[0,1]^{m}$ by
$\kappa(x)  =  \wt_{1}(\Pi(x))$
and
$\mu(x)  =  \wt_{2}(\Pi(x))$.

Generalizing the main result in~\cite{Ehrenborg_Kitaev_Perry},
we have the following theorem.
\begin{theorem}
The non-zero spectrum of the associated operator $T$
consists of discrete eigenvalues of
finite multiplicity which may accumulate only at $0$.
Furthermore, let~$r$ be a positive real number 
such that there is no eigenvalue of $T$ with modulus~$r$
and let
$\lambda_{1}, \ldots, \lambda_{k}$
be the eigenvalues of~$T$ greater in modulus than $r$.
Assume that $\lambda_{1}, \ldots, \lambda_{k}$
are simple eigenvalues with associated eigenfunctions~$\varphi_{i}$
and that the adjoint operator $T^{*}$ has
eigenfunctions $\psi_{i}$ corresponding
the eigenvalues~$\lambda_{i}$.
Then we have the expansion
\begin{equation}
    \alpha_{n}/n!
    =
\pair{T^{n-m}(\kappa)}{\mu}
    =
\sum_{i=1}^{k} 
     \frac{\pair{\varphi_{i}}{\mu} \cdot
           \pair{\kappa}{\overline{\psi_{i}}}}
          {\pair{\varphi_{i}}{\overline{\psi_{i}}}} 
     \cdot \lambda_{i}^{n-m}
     +
O(r^{n})   .
\end{equation}
\label{theorem_expansion}
\end{theorem}
The proof is the same as
in~\cite[Section~2.2]{Ehrenborg_Kitaev_Perry}
and hence omitted.

Theorem~\ref{theorem_expansion} requires us to
determine both the eigenfunction $\varphi$
and the adjoint eigenfunction~$\psi$
for each eigenvalue in order to compute the constant in each
term. However, when the weight function has symmetry
in the sense described below
then the adjoint eigenfunction can be determined from the eigenfunction.

Let $J$ be the involution on $L^{2}([0,1]^{m})$
given by
$J(f(x_{1}, x_{2}, \ldots, x_{m}))
   =
f(1-x_{m}, \ldots, 1-x_{2}, 1-x_{1})$.
Note that $J$ is a self-adjoint operator on $L^{2}([0,1]^{m})$,
that is,
$\pair{J f}{g} = \pair{f}{J g}$.
Similar to~\cite[Lemma~4.7]{Ehrenborg_Kitaev_Perry}
we have that
\begin{lemma}
Assume that the weight function $\wt$ is real-valued and satisfies
the symmetry
$$   \wt(\sigma)
   =
     \wt(m+2-\sigma_{m+1}, m+2-\sigma_{m}, \ldots, m+2-\sigma_{1})
     $$
for all $\sigma \in \SSS_{m+1}$.
If $\varphi$ is an eigenfunction of the operator $T$
with eigenvalue $\lambda$ then $\psi = J\varphi$ is an eigenfunction
of the adjoint $T^{*}$ with the eigenvalue~$\lambda$.
Furthermore, we have the equality
$\pair{f}{\overline{\psi}} = \pair{\varphi}{J f}$
for a real-valued function $f$.
\label{lemma_J}
\end{lemma}
To prove Lemma~\ref{lemma_J},
the only part that differs from the proof
in~\cite[Lemma~4.7]{Ehrenborg_Kitaev_Perry}
is the line
$\pair{f}{\overline{\psi}}
   =
 \pair{f}{\overline{J \varphi}}
   =
 \pair{f}{J \overline{\varphi}}
   =
 \pair{J f}{\overline{\varphi}}
   =
 \pair{\varphi}{\overline{J f}}
   =
 \pair{\varphi}{J f}$.

\section{Weighted descent pattern avoidance}

We now introduce weighted descent pattern avoidance
and the connection with consecutive pattern avoidance.
For a permutation
$\pi = \pi_{1} \pi_{2} \cdots \pi_{n} \in \SSS_{n}$
define its {\em descent word}
(see for
instance~\cite{Ehrenborg_Levin_Readdy,Stanley_EC_I_second})
to be
$u(\pi) = u_{1} u_{2} \cdots u_{n-1}$
where $u_{i} = \av$ if $\pi_{i} < \pi_{i+1}$
and $u_{i} = \bv$ if $\pi_{i} > \pi_{i+1}$,
that is, an $\av$ at position $i$ encodes that
$\pi$ has an ascent at position $i$ and a $\bv$
encodes a descent.

Let $\wt$ be a weight function on $\ab$-words
of length $m$, that is, the set $\{\av,\bv\}^{m}$.
Similarly, let
$\wt_{1}$ and $\wt_{2}$ be weight functions on $\ab$-words
of length $m-1$.
We extend this weight function to
words of length $n$ greater than $m-1$ by
letting
$$  \Wt(v_{1} \cdots v_{n})
   =
    \wt_{1}(v_{1} \cdots v_{m-1})
      \cdot
    \prod_{i=1}^{n-m+1}
        \wt(v_{i} \cdots v_{i+m-1})  
      \cdot
    \wt_{2}(v_{n-m+2} \cdots v_{n})   .  $$
Finally, we extend the weight to permutations
by letting
$\Wt(\pi) = \Wt(u(\pi))$.

Recall that the word $x$ has the word $w$ as a {\em factor}
if we can write $x = v \cdot w \cdot z$, where $v$ and~$z$
are also words and the dot denotes concatenation.
Let $U$ be a collection of $\ab$-words of length~$m$,
that is, $U$ is a subset of $\{\av,\bv\}^{m}$.
Define
$S(U)$ by
$$   S(U)
   =
     \{ \sigma \in \SSS_{m+1} \:\: : \:\: u(\sigma) \in U \} . $$
It is clear that a permutation $\pi$ that avoids
the descent patterns in $U$ is equivalent to that the
permutation avoids the consecutive patterns in $S(U)$.
Hence descent pattern avoidance is a special case of
consecutive pattern avoidance.

A few examples are in order. 
\begin{example}
{\rm
$m=1$ and $U = \{\bv\}$. There is only one permutation
without any descents, namely $12 \cdots n$, and hence
$\alpha_{n} = 1$.
}
\end{example}

\begin{example}
{\rm
$m=2$ and $U = \{\av\bv\}$. This forces the permutation
to have no peaks. Hence
$\alpha_{n} = 2^{n-1}$ for $n \geq 1$.
}
\label{example_ba}
\end{example}

\begin{example}
{\rm
$m=2$ and $U = \{\av\av,\bv\bv\}$. This forces the permutation
to be alternating. Alternating permutations are enumerated
by the Euler numbers, that is,
$\alpha_{n} = 2 \cdot E_{n}$ for $n \geq 2$
and
$\alpha_{n} = 1$ for $n \leq 1$.
See for instance~\cite[Section~1.6.1]{Stanley_EC_I_second}
or
\cite[Example~1.11]{Ehrenborg_Kitaev_Perry}.
}
\end{example}

For an $\ab$-word $u$ of length $m-1$ define the descent polytope $P_{u}$ to
be the subset of the unit cube $[0,1]^{m}$ corresponding to all vectors with
descent word $u$, that is,
$$
     P_{u}
  =
     \{(x_{1}, \ldots, x_{m}) \in [0,1]^{m}
          \:\: : \:\:
        x_{i} \leq x_{i+1} \mbox{ if } u_{i} = \av
          \mbox{ and }
        x_{i} \geq x_{i+1} \mbox{ if } u_{i} = \bv \} . $$
Observe that the unit cube $[0,1]^{m}$ is the union
of the $2^{m-1}$ descent polytopes.
Similar
to~\cite[Proposition~4.3 and Corollary~4.4]{Ehrenborg_Kitaev_Perry}
we have the next proposition.
Furthermore, 
the proof is also similar and hence omitted.
\begin{proposition}
Let $T$ be the operator associated with a weighted
descent pattern avoidance
and~$k$ an integer such that $0 \leq k \leq m-1$.
Let $u$ be an $\ab$-word of length $m-1$ and $f$ a function
in $L^{2}([0,1]^{m})$. Then the function
$T^{k}(f)$ restricted to the descent polytope $P_{u}$
only depends on the variables $x_{1}$ through $x_{m-k}$.
\label{proposition_restricted}
\end{proposition}
A direct consequence 
of Proposition~\ref{proposition_restricted}
is that the eigenfunctions
have a special form:
\begin{corollary}
If $\varphi$ is an eigenfunction of $T$ associated
to a non-zero eigenvalue, then the eigenfunction
$\varphi$ restricted to any descent polytope $P_{u}$
only depends on the variable $x_{1}$.
\label{corollary_x_1}
\end{corollary}

Let $V$ be the subspace of $L^{2}([0,1]^{m})$
consisting of all functions $f$ such that
the restriction $f|_{P_{u}}$ only depends on
the variable $x_{1}$ for all words $u$ of length $m-1$.
Let $f$ be a function in the subspace~$V$.
Then the function $T(f)$ is described as follows.
For an $\ab$-word $u$
of length $m-2$ and $y \in\{\av,\bv\}$ we have
\begin{equation}
\left. T(f) \right| _{P_{u y}}
      = 
\int_{0}^{x_{1}} \wt(\av u y) \cdot
   \left. f(t) \right|_{P_{\av u}}  dt 
      +
\int_{x_{1}}^{1} \wt(\bv u y) \cdot
   \left. f(t) \right|_{P_{\bv u}}  dt .
\label{equation_descent_operator}
\end{equation}
In light of Corollary~\ref{corollary_x_1}
to solve 
the eigenvalue problem
for the operator
$T : L^{2}([0,1]^{m}) \longrightarrow L^{2}([0,1]^{m})$,
it is enough to solve the eigenvalue problem
for the restricted operator $T|_{V} : V \longrightarrow V$.
The restricted operator is of a particular form,
which we describe in the next section.

\section{A general operator and its spectrum}

Recall that for a square matrix $M$ the exponential matrix of $M$ is
defined by the converging power series
$$     e^{M} = \sum_{k \geq 0} M^{k}/k!
             = I + M + M^{2}/2 + M^{3}/3! + \cdots   .   $$
The general solution of
the system of first order linear
equations $\frac{d}{dx} \vec{p}(x) = M \cdot \vec{p}(x)$
is given by
$\vec{p}(x) = e^{M \cdot x} \cdot \vec{c}$
where $\vec{c}$ is the initial condition
$\vec{p}(0)$.

Let $\gamma(M)$ denote the matrix
$$ \gamma(M) = \int_{0}^{1} e^{M \cdot t} dt  , $$
where the integration is entrywise.
Observe that
\begin{equation}
 M \cdot \gamma(M) = \int_{0}^{1} M \cdot e^{M \cdot t} dt
                   = \left[e^{M \cdot t}\right]_{0}^{1}
                   = e^{M} - I .
\label{equation_gamma_relation}
\end{equation}
Hence when $M$ is non-singular
we can write $\gamma(M) = M^{-1} \cdot(e^{M} - I)$.
Also note that by integrating the power series
of $e^{M \cdot t}$ term by term
we obtain that
$$ \gamma(M) = \sum_{k \geq 0} M^{k}/(k+1)!
             = I + M/2 + M^{2}/3! + M^{3}/4! + \cdots   .   $$

\begin{lemma}
The two following indefinite integrals hold:
\begin{eqnarray*}
      \int e^{M \cdot t} \: dt
  & = &
      \gamma(M \cdot t) \cdot t + \vec{C} , \\
      \int M \cdot t \cdot e^{M \cdot t} \: dt
  & = &
      t \cdot e^{M \cdot t} - \gamma(M \cdot t) \cdot t + \vec{C} .
\end{eqnarray*}
\end{lemma}
\begin{proof}
The first identity follows by integrating the power
series termwise. The second identity follows
from integrating the equality
$M \cdot t \cdot e^{M \cdot t} + e^{M \cdot t}
  = \frac{d}{dt} \left(t \cdot e^{M \cdot t}\right)$.
\end{proof}

Let $A$ and $B$ be two $k \times k$ matrices.
Consider the integral operator $T$ defined on
vector-valued functions by
\begin{equation}
     T(\vec{p}(x))
   =
     A \cdot \int_{0}^{x} \vec{p}(t) \: dt
   +
     B \cdot \int_{x}^{1} \vec{p}(t) \: dt   ,  
\label{equation_operator_on_matrix_form}
\end{equation}
where the integration is componentwise.

Observe that the restricted operator described in
equation~\eqref{equation_descent_operator}
is of the form~\eqref{equation_operator_on_matrix_form}
by letting~$A$ and $B$ be matrices indexed by
$\ab$-words of length $m-1$ and the entries be given by
$$   A_{u y, \av u} = \wt(\av u y)
    \:\:\:\: \mbox{ and } \:\:\:\:
     B_{u y, \bv u} = \wt(\bv u y)   $$
where $y \in \{\av,\bv\}$
and $u$ is an $\ab$-word of length $m-2$,
and the remaining entries of the matrices are~$0$.

The following theorem concerns
the eigenvalues and eigenfunctions of the operator
in~\eqref{equation_operator_on_matrix_form}.
\begin{theorem}
The non-zero spectrum of the operator $T$ 
is given by the set of non-zero roots
of the equation $\det(P) = 0$, where the matrix $P$ is given by
\begin{equation}
P
  =
- \lambda \cdot I
  + 
B \cdot \gamma((A-B)/\lambda)  ,
\label{equation_P}
\end{equation}
and the eigenfunctions are of the form
$\vec{p}(x)
        =
     e^{(A-B)/\lambda \: \cdot \: x} \cdot \vec{c}$,
where the vector $\vec{c}$ satisfies the equation
$P \cdot \vec{c} = 0$.
\label{theorem_operator}
\end{theorem}
\begin{proof}
Differentiate 
the eigenfunction equation $\lambda \cdot \vec{p} = T(\vec{p})$
with respect to $x$ to obtain the differential equation
$$   \frac{d}{dx} \vec{p}(x)
   =
     M \cdot \vec{p}(x)  ,  $$
where we let $M$ denote the matrix $1/\lambda \cdot (A-B)$.
This equation has the solution
$$   \vec{p}(x) = e^{M \: \cdot \: x} \cdot \vec{c}  , $$
where $\vec{c}$ is the initial condition.
Substituting the solution for the differential equation back
into the eigenfunction equation, we obtain
\begin{eqnarray*}
\lambda \cdot e^{M \: \cdot \: x} \cdot \vec{c}
  & = &
A \cdot \int_{0}^{x} e^{M \: \cdot \: t} \cdot \vec{c} \: dt
  +
B \cdot \int_{x}^{1} e^{M \: \cdot \: t} \cdot \vec{c} \: dt \\
  & = &
A \cdot \left[ 
  \gamma(M \cdot t) \cdot t
         \right]_{0}^{x} \cdot \vec{c}
  +
B \cdot \left[ 
  \gamma(M \cdot t) \cdot t
        \right]_{x}^{1} \cdot \vec{c} \\
  & = &
\left(
(A-B) \cdot \gamma(M \cdot x) \cdot x
  + 
B \cdot \gamma(M)
\right)
\cdot \vec{c}   \\
  & = &
\left(
\lambda \cdot \left(e^{M \: \cdot \: x} - I\right)
  + 
B \cdot \gamma(M)
\right)
\cdot \vec{c}   .
\end{eqnarray*}
Canceling terms we obtain
$P \cdot \vec{c} = 0$.
We can only find the non-zero vector $\vec{c}$ if
the matrix $P$ is singular, that is,
has a zero determinant.
\end{proof}

In the case when $A-B$ is non-singular
the condition in 
Theorem~\ref{theorem_operator}
can be expressed as
\begin{eqnarray*}
   0
  & = &
     \det(P) \cdot \det(M) \\
  & = &
     \det\left(
- A
  + 
B \cdot e^{(A-B)/\lambda}
\right) .
\end{eqnarray*}

\begin{theorem}
An eigenvalue $\lambda$ of the operator $T$
is simple if
its associated eigenfunction $\vec{p}(x)$ satisfies
the vector identity
\begin{equation}
    B \cdot e^{(A-B)/\lambda} \cdot \vec{p}(0) \neq 0 . 
\label{equation_c_c}
\end{equation}
\label{theorem_simple}
\end{theorem}
\begin{proof}
Assume that the eigenvalue $\lambda$ is not simple,
that is, it satisfies the generalized eigenvalue equation
$\lambda \cdot \vec{q} = T(\vec{q}) + \vec{p}$.
Differentiate this equation to obtain
$$  \lambda \cdot \frac{d}{dx} \vec{q}(x)
       =
    (A-B) \cdot \vec{q}(x)
       +
    \frac{d}{dx}\vec{p}(x)   .  $$
Again let $M=(A-B)/\lambda$. Multiply both sides
with $1/\lambda \cdot e^{-M \cdot x}$ to obtain
$$   e^{-M \cdot x} \cdot \frac{d}{dx} \vec{q}(x)
   -
     M \cdot e^{-M \cdot x} \cdot \vec{q}(x)
   =
    {1}/{\lambda} \cdot e^{-M \cdot x} \cdot \frac{d}{dx} \vec{p}(x)
. $$
This equation is equivalent to
$$   \frac{d}{dx} \left(e^{-M \cdot x} \cdot \vec{q}(x)\right)
   =
    {1}/{\lambda} \cdot M \cdot \vec{c} . $$
Hence we have the general solution
$$  \vec{q}(x)
   =
    {1}/{\lambda} \cdot e^{M \cdot x} \cdot M \cdot \vec{c} \cdot x
     +
    e^{M \cdot x} \cdot \vec{d}  , $$
where $\vec{d}$ is a constant vector.
Without loss of generality we can set $\vec{d} = 0$
since we are looking for a particular solution.
Inserting the particular solution
$1/\lambda \cdot e^{M \cdot x} \cdot M \cdot \vec{c} \cdot x$
into the generalized eigenvalue equation, we obtain
\begin{eqnarray*}
  M \cdot x \cdot e^{M \cdot x} \cdot \vec{c}
  & = &
  A/\lambda \cdot \int_{0}^{x}
  M \cdot t \cdot e^{M \cdot t} dt \cdot \vec{c}
    +
  B/\lambda \cdot \int_{x}^{1}
  M \cdot t \cdot e^{M \cdot t} dt \cdot \vec{c}
    +
   e^{M \: \cdot \: x} \cdot \vec{c}  \\
  & = &
  A/\lambda \cdot
\left[
    t \cdot e^{M \cdot t} - \gamma(M \cdot t) \cdot t
\right]_{0}^{x} \cdot \vec{c}
    +
  B/\lambda \cdot
\left[
    t \cdot e^{M \cdot t} - \gamma(M \cdot t) \cdot t
\right]_{x}^{1} \cdot \vec{c}
    +
   e^{M \: \cdot \: x} \cdot \vec{c}  \\
  & = &
  M \cdot
  \left(
    x \cdot e^{M \cdot x} - \gamma(M \cdot x) \cdot x
  \right) \cdot \vec{c}
    +
  B/\lambda \cdot
\left(
    e^{M} - \gamma(M)
\right) \cdot \vec{c}
    +
   e^{M \: \cdot \: x} \cdot \vec{c} .
\end{eqnarray*}
Canceling terms 
using the identity~\eqref{equation_gamma_relation}
and multiplying by $\lambda$ we have
\begin{eqnarray*}
0
  & = &
  B \cdot
\left(
    e^{M} - \gamma(M)
\right) \cdot \vec{c}
    +
   \lambda \cdot \vec{c} .
\end{eqnarray*}
Adding the equation $P \cdot \vec{c} = 0$ to this identity gives
us the conclusion of the theorem.
\end{proof}

\section{Two length four examples}
\label{section_length_four}

\subsection{No triple ascents, no triple descents}
\label{subsection_aaa_bbb}

Let us consider the case when we avoid
the two words $\av\av\av$ and $\bv\bv\bv$.
This is equivalent to avoiding the 
consecutive patterns $1234$ and $4321$.
In this case we have the two matrices
$$   A
   = 
     \begin{pmatrix}
       0 & 0 & 0 & 0 \\
       1 & 0 & 0 & 0 \\
       0 & 1 & 0 & 0 \\
       0 & 1 & 0 & 0
     \end{pmatrix}
              \:\: \mbox{ and } \:\:
     B
   = 
     \begin{pmatrix}
       0 & 0 & 1 & 0 \\
       0 & 0 & 1 & 0 \\
       0 & 0 & 0 & 1 \\
       0 & 0 & 0 & 0
     \end{pmatrix}    .    $$
Note the matrix $A-B$ is invertible and diagonalizable.
To simplify calculations let
$$ \tau   = \sqrt{\frac{1+\sqrt{5}}{2}}
         \:\:\:\: \text{ and } \:\:\:\:
   \sigma = \sqrt{\frac{-1+\sqrt{5}}{2}}  .  $$
That is, the four eigenvalues of the matrix $A-B$
are $\pm \sigma$ and $\pm \tau \cdot i$.

Using a computer algebra package as Maple,
we obtain that the determinant of the matrix $P$ from
Theorem~\ref{theorem_operator} expands as
\begin{eqnarray*}
     \frac{20}{\lambda^{4}}
   \cdot
     \det(P)
  & = &
         8
       + 
         \left( 3+i + \sqrt{5} \cdot (\tau + \sigma\cdot i) \right)
         \cdot
         e^{(\sigma + \tau \cdot i)/{\lambda}} \\
  &   &
       + 
         \left( 3-i + \sqrt{5} \cdot (\tau - \sigma\cdot i) \right)
         \cdot
         e^{(\sigma - \tau \cdot i)/{\lambda}} \\
  &   &
       + 
         \left( 3-i + \sqrt{5} \cdot (-\tau + \sigma\cdot i) \right)
         \cdot
         e^{(-\sigma + \tau \cdot i)/{\lambda}} \\
  &   &
       + 
         \left( 3+i + \sqrt{5} \cdot (-\tau - \sigma\cdot i) \right)
         \cdot
         e^{(-\sigma - \tau \cdot i)/{\lambda}} .
\end{eqnarray*}
Thus we obtain
\begin{proposition}
Let $\lambda_{0}$ be the largest real positive root of the equation
\begin{eqnarray}
        -8
  & = &
         \left( 3+i + \sqrt{5} \cdot (\tau + \sigma\cdot i) \right)
         \cdot
         e^{(\sigma + \tau \cdot i)/{\lambda}}
\nonumber \\
  &   &
       + 
         \left( 3-i + \sqrt{5} \cdot (\tau - \sigma\cdot i) \right)
         \cdot
         e^{(\sigma - \tau \cdot i)/{\lambda}}
\nonumber \\
  &   &
       + 
         \left( 3-i + \sqrt{5} \cdot (-\tau + \sigma\cdot i) \right)
         \cdot
         e^{(-\sigma + \tau \cdot i)/{\lambda}}
\nonumber \\
  &   &
       + 
         \left( 3+i + \sqrt{5} \cdot (-\tau - \sigma\cdot i) \right)
         \cdot
         e^{(-\sigma - \tau \cdot i)/{\lambda}}  .
\label{equation_aaa_bbb}
\end{eqnarray}
Then $\lambda_{0}$ is the largest eigenvalue (in modulus)
of the associated operator $T$
and the asymptotics of the number of permutations
without triple ascents and triple descents
is given by
$$   \alpha_{n}/n!
   =
     c \cdot \lambda_{0}^{n-3} + O(r^{n})   ,  $$
where $c$ and $r$ are two positive constants such that
$r < \lambda_{0}$.
\label{proposition_aaa_bbb}
\end{proposition}
\begin{proof}
It remains to show that the eigenvalue $\lambda_{0}$ is simple.
Observe that the de Bruijn graph
with the two directed edges
$\av\av \stackrel{\av\av\av}{\longrightarrow} \av\av$
and
$\bv\bv \stackrel{\bv\bv\bv}{\longrightarrow} \bv\bv$
removed is ergodic.
Now the conclusion follows from combining Theorems~1.7 and~4.2
in~\cite{Ehrenborg_Kitaev_Perry}.
\end{proof}

Solving equation~\eqref{equation_aaa_bbb} numerically
we obtain the three largest roots:
\begin{eqnarray*}
 \lambda_{0}  
  & = &
0.9240358576\ldots \\
 \lambda_{1,2}
  & = &
-0.2875224461\ldots \pm 0.4015233122\ldots \cdot i .
\end{eqnarray*}
Hence we have that 
$r$ is bounded below by
$|\lambda_{1,2}| = 0.4938523335\ldots$.

For the eigenvalue $\lambda = 0.9240358576\ldots$
we can solve for the vector $\vec{c}$ and we have
$$   \vec{c}
   =
     \begin{pmatrix}
             0.6536190979\ldots \\
             0.6536190979\ldots \\
             0.3815287011\ldots \\
             0
     \end{pmatrix}   .  $$
Thus we have the eigenfunction
$\varphi = e^{(A-B)/\lambda \: \cdot \: x} \cdot \vec{c}$
and adjoint eigenfunction 
$\psi = J \varphi$.
Note that when we restrict the adjoint eigenfunction
$\psi$ to a descent polytope we obtain a function only
depending on the last variable $x_{3}$.
For these two functions we calculate
\begin{eqnarray*}
\pair{\varphi}{{\mathbf 1}}
    =
\pair{{\mathbf 1}}{\overline{\psi}}
  & = &
0.6020376937\ldots ,
\\
\pair{\varphi}{\overline{\psi}}
  & = &
0.3647767214\ldots .
\end{eqnarray*}
Combining this we have the constant
$$ 
\frac{\pair{\varphi}{{\mathbf 1}} \cdot
      \pair{{\mathbf 1}}{\overline{\psi}}}
     {\pair{\varphi}{\overline{\psi}}} 
   =
0.9936198319\ldots . $$
Thus in numerical terms we have that
the asymptotics for
the number of permutations with no triple ascents
and triple descent is given by
$$    0.9936198319\ldots
         \cdot
      (0.9240358576\ldots)^{n-3}
         \cdot n!   . $$

\subsection{Avoiding isolated ascents and descents}

We next consider the case when we avoid
the two words $\av\bv\av$ and $\bv\av\bv$.
This is equivalent to avoiding the ten alternating
permutations
$1324$, $1423$, $2314$, $2413$, $3412$ and
$2143$, $3142$, $3241$, $4132$, $4231$.
In this case we have the two matrices
$$   A
   = 
     \begin{pmatrix}
       1 & 0 & 0 & 0 \\
       1 & 0 & 0 & 0 \\
       0 & 0 & 0 & 0 \\
       0 & 1 & 0 & 0
     \end{pmatrix}
              \:\: \mbox{ and } \:\:
     B
   = 
     \begin{pmatrix}
       0 & 0 & 1 & 0 \\
       0 & 0 & 0 & 0 \\
       0 & 0 & 0 & 1 \\
       0 & 0 & 0 & 1
     \end{pmatrix}    .    $$
Yet again the matrix $A-B$ is invertible and diagonalizable.
The eigenvalues are $\pm \tau$ and $\pm \sigma \cdot i$.
Similar to
Proposition~\ref{proposition_aaa_bbb}
we have:
\begin{proposition}
Let $\lambda_{0}$
be the largest real positive root of the equation
\begin{eqnarray}
        -8
  & = &
         \left( 3-i + \sqrt{5} \cdot (-\tau + \sigma\cdot i) \right)
         \cdot
         e^{(\tau + \sigma \cdot i)/{\lambda}}
\nonumber \\
  &   &
       + 
         \left( 3+i + \sqrt{5} \cdot (-\tau - \sigma\cdot i) \right)
         \cdot
         e^{(\tau - \sigma \cdot i)/{\lambda}}
\nonumber \\
  &   &
       + 
         \left( 3+i + \sqrt{5} \cdot (\tau + \sigma\cdot i) \right)
         \cdot
         e^{(-\tau + \sigma \cdot i)/{\lambda}}
\nonumber \\
  &   &
       + 
         \left( 3-i + \sqrt{5} \cdot (\tau - \sigma\cdot i) \right)
         \cdot
         e^{(-\tau - \sigma \cdot i)/{\lambda}} .
\label{equation_aba_bab}
\end{eqnarray}
Then $\lambda_{0}$ is the largest eigenvalue (in modulus)
of the associated operator $T$
and the asymptotics of the number of permutations
not having any isolated ascents or descents is
given by
$$   \alpha_{n}/n!
   =
     c \cdot \lambda_{0}^{n-3} + O(r^{n})   ,  $$
where $c$ and $r$ are two positive constants
such that $r < \lambda_{0}$.
\label{proposition_aba_bab}
\end{proposition}
The same argument as in
Proposition~\ref{proposition_aaa_bbb}
yields that the largest eigenvalue $\lambda$
is simple.
The only difference is that we consider the
de Bruijn graph with the two edges
$\av\bv \stackrel{\av\bv\av}{\longrightarrow} \bv\av$
and
$\bv\av \stackrel{\bv\av\bv}{\longrightarrow} \av\bv$
removed.

Numerically, we find the following three largest roots
to equation~\eqref{equation_aba_bab}:
\begin{eqnarray*}
 \lambda_{0}
  & = &
 0.6869765032\ldots , \\
 \lambda_{1,2}
  & = &
 0.1559951131\ldots \pm 0.5317098371\ldots \cdot i .
\end{eqnarray*}
The next largest root $\lambda_{1,2}$ bounds
$r$ from below by
$|\lambda_{1,2}| = 0.5541207686\ldots$.

Similar to Subsection~\ref{subsection_aaa_bbb}
we can obtain the numerical asymptotic
expression for the quantity~$\alpha_{n}$.
The numerical data is as follows:
$$   \vec{c}
   =
     \begin{pmatrix}
             0.4315640876\ldots \\
             0                  \\
             0.6378684967\ldots \\
             0.6378684967\ldots
     \end{pmatrix}   ,  $$
and
\begin{eqnarray*}
\pair{\varphi}{{\mathbf 1}}
    =
\pair{{\mathbf 1}}{\overline{\psi}}
  & = &
0.2798342976\ldots ,
\\
\pair{\varphi}{\overline{\psi}}
  & = &
0.0878970625\ldots .
\end{eqnarray*}
Combining this we have the constant
$$ 
\frac{\pair{\varphi}{{\mathbf 1}} \cdot
      \pair{{\mathbf 1}}{\overline{\psi}}}
     {\pair{\varphi}{\overline{\psi}}} 
   =
0.8908970548\ldots . $$
Finally, we conclude that the asymptotics for
the number of permutations with no isolated
ascents and no isolated descents is given by
$$    0.8908970548\ldots 
         \cdot
(0.6869765032\ldots)^{n-3}
         \cdot n!   . $$

\section{A weighted example of length three}
\label{section_2_1_1_0}

\newcommand{\waa}{\wt(\av\av)}
\newcommand{\wab}{\wt(\av\bv)}
\newcommand{\wba}{\wt(\bv\av)}
\newcommand{\wbb}{\wt(\bv\bv)}

Define a weight function on the set of $\ab$-words of length $2$
such that $\waa = 0$, $\wbb = 2$ and
$\wab = \wba = 1$
and the initial and final weight functions $\wt_{1}$ and $\wt_{2}$
are identical to $1$.
We are interested in understanding the sum
$$  \alpha_{n} 
       =
    \sum_{\pi \in \SSS_{n}} \Wt(\pi)  . $$
A more explicit way to write this sum is as follows
$$  \alpha_{n} 
       =
    \sum_{\pi}  2^{\bv\bv(\pi)}  , $$
where the sum is over all $123$-avoiding permutations of
length $n$
and
$\bv\bv(\pi)$ denotes the number of
double descents of $\pi$.

Let us refine the number $\alpha_{n}$ by
considering if the permutation begins with
an ascent or a descent, and similarly how
the permutation ends, that is,
we define 
$\alpha_{n}(\av,\av)$,
$\alpha_{n}(\av,\bv)$,
$\alpha_{n}(\bv,\av)$ and
$\alpha_{n}(\bv,\bv)$ for $n \geq 2$ by
$$  \alpha_{n}(x,y)
       =
    \sum \Wt(\pi)  , $$
where the sum is over all permutations
$\pi$ in $\SSS_{n}$ whose descent word $u(\pi)$
begins with the letter $x$ and ends with the letter $y$.
Note that $\alpha_{2}(x,y)$ is given by
the Kronecker delta $\delta_{x,y}$.
These quantities can also be expressed by changing the initial
and final weight functions.

By the symmetry
$\pi_{1}, \pi_{2}, \ldots, \pi_{n}
   \longmapsto
n+1-\pi_{n}, \ldots, n+1-\pi_{2}, n+1-\pi_{1}$
we have that
$\alpha_{n}(\av,\bv) = \alpha_{n}(\bv,\av)$.

First we consider the spectrum of the
associated operator.
\begin{theorem}
The only non-zero eigenvalue of the operator $T$
is $\lambda = 1$.
This is a simple eigenvalue. Furthermore,
the eigenfunction $\varphi$ and the adjoint eigenfunction $\psi$
associated with this eigenvalue are given by
$$  \varphi
  =
    e^{-x}
      \cdot
    \left\{
    \begin{array}{c l}
       1-x & \text{ if } 0 \leq x \leq y \leq 1 , \\
       2-x & \text{ if } 0 \leq y \leq x \leq 1 ,
    \end{array}
    \right.
\:\:\:\: \text{ and } \:\:\:\:
    \psi
  =
    e^{y-1}
      \cdot 
    \left\{
    \begin{array}{c l}
        y  & \text{ if } 0 \leq x \leq y \leq 1 , \\
       y+1 & \text{ if } 0 \leq y \leq x \leq 1 .
    \end{array}
    \right.    $$
\label{theorem_1}
\end{theorem}
\begin{proof}
The associated operator $T$ can be written in
the form~\eqref{equation_operator_on_matrix_form}
using the matrices
$$   A
   = 
     \begin{pmatrix}
       0 & 0 \\
       1 & 0
     \end{pmatrix}
              \:\: \text{ and } \:\:
     B
   = 
     \begin{pmatrix}
       0 & 1 \\
       0 & 2
     \end{pmatrix}    .    $$
Note that $A-B$ has eigenvalue $-1$ of algebraic
multiplicity $2$, but geometric multiplicity $1$,
that is, the Jordan form of $A-B$ consists of
one Jordan block of size $2$.
Computing the matrix $P$ we obtain
$$ 0 = \det(P)
     = \exp\left(-1/\lambda\right)
         \cdot
       \lambda
         \cdot
       \left( \lambda - 1 \right) , $$
which only has the non-zero
root $\lambda = 1$.
Furthermore for this root,
the null space of the matrix $P$ is
spanned by the vector
$$  \vec{c} = \begin{pmatrix} 1 \\ 2 \end{pmatrix} . $$
Finally, it is straightforward to verify
$B \cdot e^{M} \cdot \vec{c} \neq \vec{0}$,
hence $\lambda = 1$ is a simple eigenvalue
by
Theorem~\ref{theorem_simple}.
Moreover the eigenfunction $\varphi$ is given by
$$  \varphi
  =
    \exp\left(
    \begin{pmatrix}
       0 & -1 \\ 1 & -2
    \end{pmatrix}
           \cdot x
        \right)
      \cdot
    \begin{pmatrix} 1 \\ 2 \end{pmatrix} 
  =
    e^{-x}
      \cdot
    \begin{pmatrix}
       1-x \\ 2-x
    \end{pmatrix}     . $$
Since the weight function $\wt$ satisfies 
the symmetry in Lemma~\ref{lemma_J},
we obtain that the adjoint eigenfunction is
given by $\psi = J(\varphi)$.
\end{proof}

\begin{theorem}
The asymptotics of the sequences
$\alpha_{n}(\av,\av)$, 
$\alpha_{n}(\av,\bv)$,
$\alpha_{n}(\bv,\bv)$ and
$\alpha_{n}$
are given by
$$  \begin{array}{r c l}
\alpha_{n}(\av,\av)/n!
  & = &
e - 4 + 4/e       +  O(r^{n})  , \\
\alpha_{n}(\av,\bv)/n!
  & = &
1 - 2/e           +  O(r^{n})  , \\
\alpha_{n}(\bv,\bv)/n!
  & = &
1/e               +  O(r^{n})  , \\
\alpha_{n}/n!
  & = &
e - 2 + 1/e       +  O(r^{n})  ,
\end{array} $$
where $r$ is an arbitrary small positive real number.
\end{theorem}
\begin{proof}
Let ${\mathbf 1}_{\av}$ denote the function encoding an ascent,
that is,
${\mathbf 1}_{\av}(x,y) = 1$ if $x < y$ and $0$ otherwise.
Similarly, let
${\mathbf 1}_{\bv}$ be the function encoding a descent,
that is,
${\mathbf 1}_{\bv}(x,y) = 1$ if $x > y$ and $0$ otherwise.
Note that we have that
$J {\mathbf 1}_{\av} = {\mathbf 1}_{\av}$
and
$J {\mathbf 1}_{\bv} = {\mathbf 1}_{\bv}$.
By letting the initial function $\kappa$
and the final function $\mu$ vary over
the two functions
${\mathbf 1}_{\av}$ and ${\mathbf 1}_{\bv}$,
we obtain the constant term 
in the asymptotic expression
in Theorem~\ref{theorem_expansion}.
First we compute the inner products
\begin{eqnarray*}
\pair{\varphi}{{\mathbf 1}_{\av}}
    =
\pair{{\mathbf 1}_{\av}}{\overline{\psi}}
  & = &
1 - 2/e , \\
\pair{\varphi}{{\mathbf 1}_{\bv}}
    =
\pair{{\mathbf 1}_{\bv}}{\overline{\psi}}
  & = &
1/e , \\
\pair{\varphi}{\overline{\psi}}
  & = &
1/e ,
\end{eqnarray*}
where we used Lemma~\ref{lemma_J}
for two of the five equalities.
Hence the constants are:
\begin{eqnarray*}
\frac{\pair{\varphi}{{\mathbf 1}_{\av}} \cdot
      \pair{{\mathbf 1}_{\av}}{\overline{\psi}}}
     {\pair{\varphi}{\overline{\psi}}} 
  & = &
     e - 4 + 4/e   , \\
\frac{\pair{\varphi}{{\mathbf 1}_{\bv}} \cdot
      \pair{{\mathbf 1}_{\av}}{\overline{\psi}}}
     {\pair{\varphi}{\overline{\psi}}} 
  & = &
     1 - 2/e   , \\
\frac{\pair{\varphi}{{\mathbf 1}_{\bv}} \cdot
      \pair{{\mathbf 1}_{\bv}}{\overline{\psi}}}
     {\pair{\varphi}{\overline{\psi}}} 
  & = &
     1/e   .
\end{eqnarray*}
This proves the three first results of the theorem.
The fourth result is obtained by adding the asymptotic
expressions for
$\alpha_{n}(\av,\av)$,
$\alpha_{n}(\av,\bv)$,
$\alpha_{n}(\bv,\av)$ and
$\alpha_{n}(\bv,\bv)$.
\end{proof}

In order to study these sequences further,
we introduce the associated exponential
generating functions.
Let $F_{x,y}(z)$ denote the generating function
$$  F_{x,y}(z)
  =
    \sum_{n \geq 2} \alpha_{n}(x,y) \cdot \frac{z^{n}}{n!}  . $$
Similarly, let $F(z)$ be the generating function
for the sequence $\alpha_{n}$.

\begin{proposition}
The generating function $F_{x,y}(z)$ satisfies
the following equation:
\begin{eqnarray}
       F_{x,y}(z)
  & = &
       \delta_{x,y} \cdot \frac{z^{2}}{2!}
    +
       \delta_{x,\bv} \cdot \delta_{y,\av} \cdot 2 \cdot \frac{z^{3}}{3!}
\nonumber \\ 
  & + &
       \int_{0}^{z}
             (F_{x,\av}(w) + 2 \cdot F_{x,\bv}(w))
           \cdot
             F_{\bv,y}(w)   \: dw
\nonumber \\ 
  & + &
       \delta_{x,\av}
          \cdot
       \int_{0}^{z}
             F_{\bv,y}(w)   \: dw
\nonumber \\ 
  & + &
       \delta_{x,\bv}
          \cdot
       \int_{0}^{z}
             w \cdot F_{\bv,y}(w)   \: dw
\nonumber \\ 
  & + &
       \delta_{y,\bv}
          \cdot
       \int_{0}^{z}
             (F_{x,\av}(w) + 2 \cdot F_{x,\bv}(w))  \: dw
\nonumber \\ 
  & + &
       \delta_{y,\av}
          \cdot
       \int_{0}^{z}
             (F_{x,\av}(w) + 2 \cdot F_{x,\bv}(w)) \cdot w   \: dw .
\label{equation_generating_functions}
\end{eqnarray}
\label{proposition_generating_functions}
\end{proposition}
\begin{proof}
We demonstrate that all the terms on the right-hand side
are in fact counting permutations.
The first term
corresponds to permutations of length $2$.
The second term corresponds to permutations
of length $3$ with the element $1$ in the middle position,
that is, the two permutations $213$ and~$312$.

For the remaining permutations we break a permutation
at the position where the element~$1$ occurs. We obtain
two smaller permutations $\sigma$ and $\tau$ of lengths $k$,
respectively, $r$, where $k+r=n-1$.
The elements are distributed in $\binom{n-1}{k}$ ways
between these two permutations. This is encoded by
multiplication of exponential generating functions.
Finally, the integral shifts the coefficient
from $w^{n-1}/(n-1)!$ to~$z^{n}/n!$.

We continue to describe the terms.
The third term corresponds to $2 \leq k,r$,
that is,
at least two elements precede the element $1$
and
at least two elements follow the element $1$.
Note that $\tau$ must begin with a descent to
avoid creating a double ascent.
Also when $\sigma$ ends with a descent, we create
a double descent when concatenating $\sigma$
with the element $1$. This explains the factor $2$
in front of the term $F_{x,b}$.

The fourth term corresponds to $k=0$ and $r \geq 2$.
The Kronecker delta states that the permutation starts with an ascent.
The fifth term corresponds to $k=1$ and $r \geq 2$, in which
the permutation starts with a consecutive descent and ascent.
Similarly, the sixth and seventh terms correspond
to the two cases $r = 0$ and $k \geq 2$,
respectively, $r = 1$ and $k \geq 2$.

Since each permutation has been accounted for,
the equality holds.
\end{proof}

Note that Proposition~\ref{proposition_generating_functions}
is similar in spirit to the equations obtained
by Elizalde and Noy~\cite{Elizalde_Noy}
for the generating functions for certain
classes of pattern avoidance permutations.

\begin{theorem}
The generating functions $F_{x,y}(z)$ and $F(z)$ are given
by
\begin{eqnarray*}
F_{\av,\av}(z)
  & = &
\frac{1}{1-z} \cdot \left(e^{z} - 4 + 4 \cdot e^{-z}\right)
 - 1 + 2 \cdot z , \\
F_{\av,\bv}(z)
  & = &
\frac{1}{1-z} \cdot \left(1 - 2 \cdot e^{-z}\right)
 + 1 - z , \\
F_{\bv,\bv}(z)
  & = &
\frac{1}{1-z} \cdot e^{-z} - 1 , \\
F(z)
  & = &
\frac{1}{1-z} \cdot \left(e^{z} - 2 + e^{-z}\right) .
\end{eqnarray*}
\label{theorem_generating_functions}
\end{theorem}
\begin{proof}
Proposition~\ref{proposition_generating_functions}
can be viewed as a recursion for the coefficient
$\alpha_{n}(x,y)$. Hence the equation in
this proposition has a unique solution and it is
enough to verify the theorem by showing that
the proposed generating functions satisfy
equation~\eqref{equation_generating_functions}.

Finally, the generating function $F(z)$ is obtained
by adding the four generating functions $F_{\av,\av}(z)$,
$F_{\av,\bv}(z)$, $F_{\bv,\av}(z)$ and $F_{\bv,\bv}(z)$.
\end{proof}
Since $e^{-z}/(1-z)$ is the
generating function for the number of derangements,
we obtain
\begin{corollary}
For $n \geq 2$, the number of derangements on $n$ elements, $D_{n}$,
is given by $\alpha_{n}(\bv,\bv)$, that is,
$$
  D_{n}
      =
  \sum_{\pi} 2^{\bv\bv(\pi)} ,
$$
where the sum is over all permutations $\pi$ on $n$ elements
with no double ascents and starting and ending with a descent.
\label{corollary_derangements}
\end{corollary}

As a corollary to Theorem~\ref{theorem_generating_functions}
we have the following recursions:
\begin{corollary}
Recursions for the sequences
$\alpha_{n}(\av,\av)$, 
$\alpha_{n}(\av,\bv)$,
$\alpha_{n}(\bv,\bv)$ and
$\alpha_{n}$
are given by,
where $n \geq 3$,
\begin{eqnarray*}
\alpha_{n}(\av,\av)
  & = &
    n \cdot \alpha_{n-1}(\av,\av)
                +
    1 + 4 \cdot (-1)^{n} , \\
\alpha_{n}(\av,\bv)
  & = &
    n \cdot \alpha_{n-1}(\av,\bv)
                -
    2 \cdot (-1)^{n} , \\
\alpha_{n}(\bv,\bv)
  & = &
    n \cdot \alpha_{n-1}(\bv,\bv)
                +
    (-1)^{n} , \\
\alpha_{n}
  & = &
    n \cdot \alpha_{n-1}
                +
    1 + (-1)^{n} .
\end{eqnarray*}
\end{corollary}

Using the generating functions in
Theorem~\ref{theorem_generating_functions}
we now obtain that the error terms are
the smallest possible.
We express the result as 
explicit expressions using the nearest integer
function, which we denote by~$\nint{x}$.
\begin{theorem}
The quantities 
$\alpha_{n}(\av,\av)$, 
$\alpha_{n}(\av,\bv)$,
$\alpha_{n}(\bv,\bv)$ and
$\alpha_{n}$
are given by
the explicit expressions
$$  \begin{array}{r c c c}
\alpha_{n}(\av,\av)
  & = &
\nint{(e - 4 + 4/e) \cdot n!} 
  & \text{ for } n \geq 8, \\
\alpha_{n}(\av,\bv)
  & = &
\nint{(1 - 2/e) \cdot n!} 
  & \text{ for } n \geq 3, \\
\alpha_{n}(\bv,\bv)
  & = &
\nint{1/e \cdot n!} 
  & \text{ for } n \geq 2, \\
\alpha_{n}
  & = &
\nint{(e - 2 + 1/e) \cdot n!} 
  & \text{ for } n \geq 4. \\
\end{array} $$
\label{theorem_exact}
\end{theorem}
\begin{proof}
The third equality is classical.
We show the first equality.
The coefficient of $z^{n}/n!$ in the generating function
$F_{\av,\av}(z)$, for $n \geq 2$, is given by
$$
\alpha_{n}(\av,\av)
    = 
n! 
   \cdot
\sum_{k=0}^{n}
   \frac{1^{k} - 4 \cdot 0^{k} + 4 \cdot (-1)^{k}}{k!} . $$
Hence the difference
$$   n! \cdot (e - 4 + 4/e)
        -
     \alpha_{n}(\av,\av)
  =
n! 
   \cdot
\sum_{k \geq n+1}
   \frac{1^{k} + 4 \cdot (-1)^{k}}{k!}  , $$
is bounded in absolute value by
$$
n!
   \cdot
\sum_{k \geq n+1}
   \frac{5}{k!} 
    =
\frac{5}{n+1}
  + 
\frac{5}{(n+1) \cdot (n+2)}
  +
\cdots . $$
Note that this is a decreasing function in $n$.
For $n=10$ this function dips below $1/2$, showing
the first equality for $n \geq 10$. The two cases
$n=8,9$ can be done by hand. 
The second and fourth equalities follow by similar arguments.
\end{proof}

\section{Concluding remarks}

Are there other operators of the form~\eqref{equation_T}
which only have a finite number of non-zero eigenvalues?
Furthermore, if the associated sequences are
integer sequences would the corresponding error term
be the smallest possible,
as in Theorem~\ref{theorem_exact}?

The operators of the form~\eqref{equation_T}
have so far yielded four types of behavior:
\begin{itemize}
\item[(i)]
The operator has an infinite number of eigenvalues
and the asymptotic expansion converges.
An example of this is alternating permutations.
See~\cite[Example~1.11]{Ehrenborg_Kitaev_Perry}
and~\cite{Ehrenborg_Levin_Readdy}.
Another example is $\{123, 231, 312\}$-avoiding
permutations. See~\cite[Section~7]{Ehrenborg_Kitaev_Perry}.

\item[(ii)]
The operator has an infinite number of eigenvalues
and the asymptotic expansion does not give
an expression that converges.
This occurs with 
$123$-avoiding permutations
and $213$-avoiding permutations.
See~\cite[Sections~5 and~6]{Ehrenborg_Kitaev_Perry}.

\item[(iii)]
The operator has a finite, but positive, number of non-zero eigenvalues.
See Section~\ref{section_2_1_1_0}.
For instance, is there such an operator with exactly
two non-zero eigenvalues?
What behavior does the error term of the asymptotic expansion
have? Are there other examples with the smallest possible
error term?

\item[(iv)]
The operator has no non-zero eigenvalues.
Here the behavior can vary a lot.
Compare $\bv\av$-avoiding permutations
in Example~\ref{example_ba}
with $\{312, 321\}$-avoiding permutations
in~\cite{Claesson}.
Also see~\cite[Example~3.9]{Ehrenborg_Kitaev_Perry}.
\end{itemize}

The two equations~\eqref{equation_aaa_bbb}
and~\eqref{equation_aba_bab} in
Section~\ref{section_length_four}
have an interesting pattern in their roots.
Consider the two equations in terms of the variable
$z = 1/\lambda$. Then the roots lie on the real axis
and close to a vertical line in the complex plane.
Is there an explanation for this behavior?
Switching back to the variable $\lambda$
it says that the roots lie on the real axis
and close to a circle in the complex plane.

Baxter, Nakamura and Zeilberger~\cite{Baxter_Nakamura_Zeilberger}
have developed efficient methods to compute
the number of permutations avoiding certain patterns.
Their methods use umbral techniques
and have been implemented in Maple.
Their techniques can be extended to compute
the weighted problem introduced in this paper.

In Section~\ref{section_2_1_1_0}
we obtained that
the number of derangements $D_{n}$
is given by the sum over all permutations
with no double ascents, where each term
is $2$ to the power of the number of
double descents.
It is natural to ask for a bijective proof
of this fact. In fact, 
Muldoon Brown and Readdy~\cite{Muldoon_Readdy}
gave essentially such a bijection.

A {\em descent run} $I$
in a permutation $\pi = \pi_{1} \pi_{2} \cdots \pi_{n}$
is an interval $I = [i,j]$ such that
$\pi_{i} > \pi_{i+1} > \cdots > \pi_{j}$.
Observe that we do not require the interval to be maximal.

Given a permutation $\pi$ on $n$ elements,
write the interval $[1,n]$ as a disjoint union
$\bigcup_{r=1}^{k} I_{r}$ of
descent runs of the permutation $\pi$.
A Muldoon--Readdy barred permutation $\sigma$ from $\pi$
is obtained by the following procedure. For
each descent run $I_{r} = [i,j]$ pick an
element $h_{r}$ in the half-open interval
$(i,j]$ and bar the elements $\pi_{h_{r}}$ through
$\pi_{j}$.
Observe that for a descent run $I_{r}$ of cardinality~$c$
there are $c-1$ possibilities to pick the element $h_{r}$.
Hence to obtain Muldoon--Readdy barred permutation
each descent run needs to have cardinality at least $2$.

Finally given a permutation $\pi$, how many
Muldoon--Readdy barred permutations
can be obtained from it?
If the permutation $\pi$ has a double ascent
there is a maximal descent run of size $1$
and hence there is no way to obtain
a Muldoon--Readdy permutation.
Partition the interval $[1,n]$ into
maximal descent runs of $\pi$,
that is, $[1,n] = \bigcup_{s} J_{s}$.
Each maximal descent run $J_{s}$ can be further
partitioned into descent runs and be barred.
Let the maximal descent run $J_{s}$ have cardinality~$c$.
Then the number of ways of partitioning $J_{s}$
and barring each descent run is
$$    \sum_{(c_{1}, c_{2}, \ldots, c_{k})}
           (c_{1}-1) \cdot (c_{2}-1) \cdots (c_{k}-1)  , $$
where the sum ranges over all compositions of the integer $c$.
By the basic generating function argument
$1/(1 - x^{2}/(1-x)^{2}) = 1 + x^{2}/(1-2x)$,
the above sum is $2^{c-2}$.
Finally, note that $c-2$ is the number
of double descents in the descent run $J_{r}$.
Hence there are
$2$ to the number of double descents in $\pi$
ways to obtain a Muldoon--Readdy barred permutation from
the permutation $\pi$.

Finally,
Muldoon Brown and Readdy give
a bijection between 
derangements and 
Muldoon--Readdy barred permutations.
See Theorem~6.4 in~\cite{Muldoon_Readdy}.
This gives a bijective proof of
Corollary~\ref{corollary_derangements}.

\section*{Acknowledgments}

The authors thank
Margaret Readdy and the two referees for their comments
on an earlier draft of this paper.
The authors are partially funded by
the National Science Foundation grant DMS-0902063.
The first author also thanks the Institute for Advanced
Study and is also partially supported by 
the National Science Foundation grants
DMS-0835373 and CCF-0832797.

\newcommand{\journal}[6]{{\sc #1,} #2, {\it #3} {\bf #4} (#5), #6.}
\newcommand{\preprint}[3]{{\sc #1,} #2, preprint #3.}
\newcommand{\appear}[3]{{\sc #1,} #2, to appear in {#3.}}
\newcommand{\JCTA}{J.\ Combin.\ Theory Ser.\ A}
\newcommand{\book}[4]{{\sc #1,} ``#2,'' #3, #4.}

\bigskip

{\em R.\ Ehrenborg and J.\ Jung,
Department of Mathematics,
University of Kentucky,
Lexington, KY 40506-0027,}
\{{\tt jrge},{\tt jjung}\}{\tt@ms.uky.edu}

\end{document}